\documentclass{amsart}
\usepackage{amssymb}

\usepackage{amsfonts}
\usepackage[all,2cell]{xy} \UseAllTwocells \SilentMatrices
\usepackage{graphicx}
\usepackage{amsthm}
\usepackage{amstext}
\usepackage{amsmath}
\usepackage{amscd}
\usepackage[mathscr]{eucal}
\usepackage{url}

\newtheorem*{theorem*}{Theorem}
\newtheorem{theorem}{Theorem}[section]
\newtheorem{proposition}[theorem]{Proposition}
\newtheorem{corollary}[theorem]{Corollary}
\newtheorem{lemma}[theorem]{Lemma}

\theoremstyle{definition}
\newtheorem{definition}[theorem]{Definition}
\newtheorem{example}[theorem]{Example}

\theoremstyle{remark}
\newtheorem{remark}[theorem]{Remark}

\numberwithin{equation}{section}

\newcommand\mL{L\kern-0.08cm\char39}

\begin{document}

%

\title{Persistent Shadowing For Actions Of Some Finitely Generated Groups And Related Measures}
\author { Ali Barzanouni}
\address{Department of Mathematics, School of Mathematical Sciences, Hakim Sabzevari University, Sabzevar, Iran}
\email{a.barzanouni@hsu.ac.ir, barzanouniali@gmail.com} \subjclass[2010]{Primary: 37C85; Secondary: 37B25, 37B05}

\keywords{ Shadowing, Persistent, Borel Measure}
\date{}
\maketitle

\begin{abstract}
In this paper, $\varphi:G\times X\to X$ is a continuous action of
finitely generated group $G$ on compact metric space $(X, d)$
without isolated point. We introduce the notion of persistent
shadowing property for  $\varphi:G\times X\to X$ and study it via measure theory.  Indeed, we introduce the notion of compatibility the  Borel probability measure
$\mu$ with respect persistent shadowing property of $\varphi:G\times
X\to X$ and denote it by $\mu\in\mathcal{M}_{PSh}(X, \varphi)$. We
show $\mu\in\mathcal{M}_{PSh}(X, \varphi)$ if and only if
$supp(\mu)\subseteq PSh(\varphi)$, where $PSh(\varphi)$ is the set of all persistent shadowable points of $\varphi$. This implies that if every non-atomic Borel probability measure
$\mu$ is compatible with persistent shadowing property for
$\varphi:G\times X\to X$, then $\varphi$ does have persistent
shadowing property.  We prove that
$\overline{PSh(\varphi)}=PSh(\varphi)$ if and only if
$\overline{\mathcal{M}_{PSh}(X, \varphi)}= \mathcal{M}_{PSh}(X,
\varphi)$. Also, $\mu(\overline{PSh(\varphi)})=1$ if and only if
$\mu\in\overline{\mathcal{M}_{PSh}(X, \varphi)}$. Finally, we show
that $\overline{\mathcal{M}_{PSh}(X, \varphi)}=\mathcal{M}(X)$ if
and only if $\overline{PSh(\varphi)}=X$.
 For study of persistent
shadowing property, we introduce the notions of uniformly $\alpha$-persistent point,  uniformly $\beta$-persistent point  and  recall notions of shadowing property,
$\alpha$-persistent,  $\beta$-persistent and we give some further
results about them.

\end{abstract}
\section{Introduction}
A continuous action  is a triple  $(X, G, \varphi)$  where $X$ is a
compact metric space  with a metric $d$ and  $G$ is  a finitely
generated group with the discrete topology which acts on $X$, such
that the action $\varphi$ is continuous.
  We denote a continuous action $(X, G, \varphi)$ by $\varphi:G\times X\to X$, also we denote by $Act(G;X)$ the set of all
continuous actions $\varphi$ of $G$ on $X$. Let  $S$ be a finite
generating set of $G$.  We consider a metric $d_S$ on $Act(G;X)$ by

\begin{equation*}
d_S(\varphi, \psi)= \sup \{d(\varphi(s, x), \psi(s, x)): x\in X,
s\in S\}
\end{equation*}
for $\varphi, \psi\in Act(G; X)$.\\
A map $f:G\to X$ is called a  $\delta$-pseudo orbit for a continuous
action $\varphi:G\times X\to X$ ( with respect to $S$), if $d(f(sg),
\varphi(s, f(g)))<\delta$ for all $s\in S$ and all $g\in G$. A
$\delta$-pseudo orbit $f:G\to X$ (with respect to $S$ ) is
$\epsilon$-shadowed by $\varphi$-orbit $x\in X$, if $d(f(g),
\varphi(g, x))<\epsilon$ for all $g\in G$.  A continuous action
$\varphi:G\times X\to X$ has shadowing property (with respect to
$S$) if for every $\epsilon>0$ there is a $\delta>0$ such that every
$\delta$-pseudo orbit $f:G\to X$ for $\varphi$ can be $
\epsilon$-shadowed by the $\varphi$-orbit of a point $p\in X$, this
means that $d(f(g), \varphi(g, p))<\epsilon$ for all $g\in G$. The
notion of shadowing property for actions of finitely generated
groups introduced by Osipov and Tikhomirov in \cite{osipov}. They
showed that the shadowing property for actions of finitely generated
groups depends  on both of the  hyperbolic properties of actions of
its elements and  the group structure. For example, if $G$ is a
finitely generated nilpotent group and action of one element in $G$
is hyperbolic, then the group action has the shadowing property
while it cannot be directly generalized to the case of solvable
groups.

 The notion of topological stability for an action
of a finitely generated group on a compact metric space was
introduced by
 Chung
and Lee in  \cite{chung} and they gave a group action version of the
Walter's stability theorem.  Indeed,  a continuous action
$\varphi:G\times X\to X$ is topologically stable (with respect to
$S$) , if for every $\epsilon>0$ there is $\delta>0$ such that for
every continuous action   $\psi:G\times X\to X$  with $d_S(\varphi,
\psi)<\delta$,  there is a continuous map $f:X\to X$ such that
 $d_{C^0}(f, id)<\epsilon$ and $\varphi_gf=f\psi_g$ for all $g\in G$. Moreover, $\varphi$ is called $s$-topologically stable when there exists a surjective  continuous map $f:X\to X$ that satisfies  the mentioned properties.
  If $\varphi:G\times X\to X$ is topologically stable, then  for every $\epsilon>0$ there is $\delta>0$ such that for every $x\in X$ and every continuous action $\psi:G\times X\to X$ with $d_S(\varphi, \psi)<\delta$,  we have $d(\varphi(g, f(x)), \psi(g, x))<\epsilon$ for all $g\in G$. Having this  property is well known to say that the  continuous action $\varphi$ is $\alpha$-persistent. When $\varphi$ is $s$-topologically stable,  for every $\epsilon>0$ there is $\delta>0$ such that for every $x\in X$ and every continuous action $\psi:G\times X\to X$ with $d_S(\varphi, \psi)<\delta$,  we can say that if $f(y)=x$, then  $d(\varphi(g, x ), \psi(g, y))<\epsilon$ for all $g\in G$.
  In this case,  $\varphi$ is called $\beta$-persistent. In other words, a dynamical
system is $\beta$- persistent if its trajectories can be seen on
every small perturbation of it.
Although $s$-topologically stable implies
$\beta$-persistent but topological stability does not imply
$\beta$-persistent. For example, Sakai and Kobayashi \cite{sakai1}
observed that the full shift on two symbols is not
$\beta$-persistent while it is topologically stable. Recently, the
authors in \cite{jung},  introduced a new
tracing property for a homeomorphism $f:X\to X$ referred to as persistent shadowing property and proved that a homeomorphism has persistent shadowing property
 if and only if it has shadowing property and it is $\beta$-persistent. This implies that a homeomorphism has persistent shadowing property if and only if it
  is pointwise persistent shadowable.\\

  \medskip
\noindent

In this paper, we extend   the notion of persistent shadowing
property for a continuous action $\varphi:G\times X\to X$ of some
finitely generated group $G$ on metric space $(X, d)$. Persistent shadowing property is stronger than
of shadowing property and $\beta$-persistent, but in equicontinuous
actions, shadowing and persistent shadowing properties are
equivalent. This implies that every equicontinuous action on the
Cantor space $X$ does have persistent shadowing property. The notion
of persistent shadowing property does  not depend on the choice of a
symmetric finitely generating set and it does not depend on choice
of metric $X$ if $X$ is compact metric space. But Example
\ref{example2}  shows that compactness is essential. Assume that $H$
be a subgroup of $G$. It may be happen that $\varphi:H\times X\to X$
does have the persistent shadowing property
         while $\varphi:G\times X\to X$ does not have it. But in Proposition \ref{syndetic}, we show that if $H$ is a syndetic subgroup of $G$, then the situation
         is different. Also we study relation between persistent shadowing property of $\varphi:G\times X\to X$ and $\varphi_g:X\to X$.
          There is system $\varphi:G\times X\to X$ with persistent shadowing property while $\varphi_g:X\to X$ does not have persistent shadowing property.
         If $G$ is free group, then the situation is different. Indeed in Proposition \ref{op2},
         we show that if $F_2=\langle a, b\rangle$ is a free group and $\varphi:F_2\times X\to X$ has shadowing property, then $\varphi_{a^{-1}b}:X\to X$ has
         persistent shadowing property. Also, one can check that these results do hold for notions of shadowing property, $\alpha$-persistent and $\beta$-persistent, see Remark\ref{224} and Remark \ref{225}\\
\medskip
\noindent

 Recently, in \cite{ali2}, we introduced the notion of
compatibility of a measure with respect to $\alpha$-persistent. We extend this notion with respect persistent shadowing property and the set of
compatibility measures with persistent shadowing property for
$\varphi$ is denoted by $\mathcal{M}_{PSh}(X, \varphi)$, see  Subsection \ref{s22}. We show that
$\mathcal{M}_{PSh}(X, \varphi)$ is an $F_{\sigma\delta}$ subset of
$\mathcal{M}(X)$ and for $\mu\in\mathcal{M}_{PSh}(X, \varphi)$ and
homeomorphism $f:X\to Y$, we have  $f_*(\mu)\in
\mathcal{M}_{PSh}(Y, f\circ \varphi\circ f^{-1})$  where $f\circ
\varphi\circ f^{-1}:G\times Y\to Y$ is defined by
        $f\circ \varphi\circ f^{-1}(g, x)=f\circ \varphi_g\circ f^{-1}(x)$, see Proposition \ref{kj}. In Proposition \ref{pki}, we show that
         if measure $\mu$ is compatible with persistent shadowing property  for continuous action $\varphi$, then $\varphi$ does have persistent shadowing property on
          $supp(\mu)$. This implies that if every non-atomic probability measure is compatible with persistent shadowing property  for continuous action $\varphi$,
         then $\varphi$ has persistent shadowing property. Also, we introduce compatibility a measure with respect shadowing property, $\alpha$-persistent and
         $\beta$-persistent for continuous action $\varphi:G\times X\to X$ and denote them by
         $\mathcal{M}_{Sh}(X, \varphi)$, $\mathcal{M}_\alpha(X, \varphi)$ and $\mathcal{M}_{\beta}(X, \varphi)$, respectively.  Results of Proposition \ref{pki}
          can be obtain for compatibility a measure in the case of shadowing property and
$\beta$-persistent, see Remark \ref{pkii}.

         In Section 3, we introduce the notions of   persistent shadowable points, uniformly $\alpha$-persistent point, uniformly $\beta$-persistent point
         and denote them by $PSh(\varphi)$, $UPersis_\alpha(\varphi)$ and $UPersis_\beta(\varphi)$, respectively. Also, we recall  the notions of shadowable points,
          $\alpha$-persistent points, $\beta$-persistent points for continuous action
$\varphi:G\times X\to X$ and denote them by $Sh(\varphi)$,
$Persis_\alpha(\varphi)$ and $Persis_\beta(\varphi)$, respectively.
Although $Sh(\varphi)\subseteq UPersis_\alpha(\varphi)\subseteq
Persis_\alpha(\varphi)$ but Example \ref{non-shadowable} shows that
the $Sh(\varphi)\neq UPersis_\alpha(\varphi)$ and
$Upersis_\alpha(\varphi)\neq Persis_\alpha(\varphi)$. For
equicontinuous action $\varphi:G\times X\to X$, we have
$UPersis_\beta(\varphi)=Persis_\beta(\varphi)$ and
$Persis_\alpha(\varphi)\subseteq Persis_\beta(\varphi)$. Moreover,
if $X$ is generalized homogeneous compact metric space, then
$Sh(\varphi)=UPersis_\alpha(\varphi)=Persis_\alpha(\varphi)$, see Proposition \ref{u}.\\
In Subsection \ref{400}, we study persistent shadowable point for a
group action and in item 3 of Proposition \ref{wok}, we show that
   Continuous action $\varphi:G\times X\to X$  has the persistent shadowing property if and only if it is pointwise persistent shadowable. Also in item 4 of
    Proposition \ref{wok}, we prove that $PSh(\varphi)= UPersis_\beta(\varphi)\cap Sh(\varphi)$. This implies that continuous action $\varphi:G\times X\to X$ has
     persistent shadowing property if and only if it is $\beta$-persistent and it has shadowing
     property.\\
    In Subsection \ref{401}, we study  various shadowable points via
    measure theory. Indeed,  for continuous action $\varphi:G\times X\to X$ on compact metric space $(X,
    d)$  and Borel probability measure
    $\mu$,   we show that
 $\mu\in M_{PSh}(X, \varphi)\Leftrightarrow
    supp(\mu)\subseteq PSh(\varphi)$, see Proposition \ref{Lb}. In Proposition \ref{Lbb} we show that
$\mu(\overline{PSh(\varphi))}=1\Leftrightarrow
\mu\in\overline{\mathcal{M}_{PSh}(X,
    \varphi)}$ and in Proposition \ref{12354}, we show that $\overline{PSh(\varphi)}=PSh(\varphi)$ if and only if
$\overline{\mathcal{M}_{PSh}(X, \varphi)}=\mathcal{M}_{PSh}(X,
\varphi)$. Note that, result of this paragraph, obtain for other
types of shadowing property, see Proposition \ref{12355}. In
equicontinuous action $\varphi:G\times X\to X$,
$Persis_\beta(\varphi)=UPersis_\beta(\varphi)$ is closed set in $X$,
hence we can say that $\overline{\mathcal{M}_{\beta}(X,
\varphi)}=\mathcal{M}_\beta(X, \varphi)$ if $\varphi:G\times X\to X$
is equicontinuous action. This implies that
$\mu(Persis_\beta(\varphi))=1$ if and only if
$\mu\in\mathcal{M}_\beta(X, \varphi)$, whenever $\varphi:G\times
X\to X$ is equicontinuous action.
 Finally, In Proposition  \ref{3214}, we show that
$\overline{\mathcal{M}_{PSh}(X, \varphi)}=\mathcal{M}(X)
    \Leftrightarrow \overline{PSh(\varphi)}=X$, $\overline{\mathcal{M}_{Sh}(X, \varphi)}=\mathcal{M}(X)
    \Leftrightarrow \overline{Sh(\varphi)}=X$,  $\overline{\mathcal{M}_\beta(X, \varphi)}=\mathcal{M}(X)
    \Leftrightarrow \overline{Persis_\beta(\varphi)}=X$ and  $\overline{\mathcal{M}_{\alpha}(X, \varphi)}=\mathcal{M}(X)
    \Leftrightarrow \overline{Persis_\alpha(\varphi)}=X$.

\section{Persistent shadowing property }
In this section, firstly, we extend  the notion of persistent shadowing property  from \cite{jung} to group actions and study it.
     \begin{definition}\label{def21}
 A continuous action $\varphi:G\times X\to X$ has persistent shadowing property ( with respect $S$) if for every $\epsilon>0$ there is $\delta>0$ such that
  every $\delta$-pseudo orbit $f:G\to X$ for $\psi:G\times X\to X$ with $d_S(\varphi, \psi)<\delta$ can be $(\psi, \epsilon)$-shadowed by a point $p\in X$.

\end{definition}
It is not hard to see that  notion of persistent shadowing property
does  not depend on the choice of a symmetric finitely generating
set. Also one can check that the this notion does not depend on
choice of metric $X$ if $X$ is compact metric space. The following
example shows that compactness is essential.

\begin{example}\label{example2}
Let $T:\mathbb{R}\to S^1\setminus\{(0, 1)\}$ be a map given by
\begin{equation*}
T(t)=(\frac{2t}{1+t^2}, \frac{t^2-1}{t^2+1}), \quad \text{for all
}t\in\mathbb{R},
\end{equation*}
and let $X=T(\mathbb{Z})$. Let $d'$ be the metric on $X$ induced by
the Riemannian metric on $S^1$, and let $d$ be a discrete metric on
$X$. It is clear that $d$ and $d'$ induce the same topology on $X$.
Let $g_1 :X\to X$ be a homeomorphism defined by $g(a_i)=a_{i+1}$ and
$g_2:X\to X$ be defined by $g_2(a_i)= a_{i+2}$. Consider  action
$\varphi:G\times X\to X$ generated $g_1, g_2:X\to X$. Since the
metric $d$ is discrete, one can see that $\varphi$ does have
persistent shadowing property. By contradiction, let $\varphi$ has
persistent shadowing property with respect to $ d' $. Hence it does
have shadowing property with respect to $d'$.


 For $\epsilon=\frac{1}{2}$ and $z\in X$, let $\delta>0$ be
an $ \epsilon $-modulus of  shadowing property of  continuous action
$\varphi$. Choose $k\in\mathbb{N}$ satisfying $d'(a_k,
a_{-k})<\frac{\delta}{2}$, and consider a homeomorphisms $f_i:X\to
X$, $i=1, 2$, given by
 \[ f_1(a_i)= \left\lbrace
  \begin{array}{c l}
     a_{i+1}, & \text{\rm{ $ i\in\{-k, \ldots, k-1\}$}},\\
a_{-k}, & \text{\rm{$i=k$}},\\
a_i, & \text{\rm{otherwise}}.
  \end{array}
\right. \] and
 \[ f_2(a_i)= \left\lbrace
  \begin{array}{c l}
     a_{i+2}, & \text{\rm{ $ i\in\{-k, \ldots, k-2\}$}},\\
     a_{-k+1}, &\text{\rm{$i=k-1$}}\\
a_{-k}, & \text{\rm{$i=k$}},\\
a_i, & \text{\rm{otherwise}}.
  \end{array}
\right. \]

 Since $d'(f_i(x), g_i (x))<\delta$ for all
$x\in X$, hence if $\psi:G\times X\to X$ generated by $f_1, f_2$,
then $d(\varphi, \psi)<\delta$ and for $x\in X$   there is $z\in X$
such that $d'(\varphi(g, z), \psi(g, x))<\epsilon$, this implies
that  $d'(g_{i}^n(z), f^n_i(y))<\epsilon$. Since $\{g_1^n(z), n\in
\mathbb{Z}\}=X$, so we can find an integer $ n\in \mathbb{Z}$ such
that $d'(g^n_1(z),f^n_1(y))\geq\epsilon$, which is a contradiction.
Therefore $\varphi$ does have persistent shadowing property with
respect to $ d $ but it does not have  persistent shadowing property
with respect to $ d' $.
\end{example}

\subsection{The Action of Syndetic Subgroups}\label{s20}

Let $BS(1, n)=\langle a, b: ba = a^nb \rangle$ and $\varphi:G\times
\mathbb{R}^2\to \mathbb{R}^2$ be generated by $f_a(x)=Ax$ and
$f_b(x)=Bx$ where
\begin{equation}
A=\left(%
\begin{array}{cc}
  1 & 0 \\
  1 & 1 \\
\end{array}%
\right)  \;\& \;\;
B=\left(%
\begin{array}{cc}
  \lambda & 0 \\
  0 & n\lambda \\
\end{array}%
\right)
\end{equation}
Then for $1<\lambda\leq n$ and $n>1$, $f_b$ has persistent shadowing
property. In \cite{jung}, it is shown that the homeomorphism $f:X\to
X$ does have persistent shadowing property if and only if it does
have shadowing property and it is $\beta$-persistent. Hence if
$H=\langle b\rangle$ is a subgroup of $B(1, n)$, then
$\varphi|H:H\times \mathbb{R}^2\to \mathbb{R}^2$ has persistent
shadowing property while by \cite[Theorem 4.4(1)]{osipov},
$\varphi:G\times \mathbb{R}^2\to \mathbb{R}^2$ does not have
shadowing property.
In this subsection, we show that if $H$ is a syndetic subgroup of $G$, then situation is different.\\
 Let $||g||_{S}$ denote the length of the shortest representation  of the element $g$ in term of element from $S$. For continuous action $\varphi:G\times X\to X$   on compact metric space $(X,d)$, $\epsilon>0$ and $k\in\mathbb{N}$, there is $\delta>0$ such that for all $g\in G$ with $||g||_S<k$
\begin{align}\label{u}
d(x, y)<\delta\Rightarrow d(\varphi(g, x), \varphi(g,
y))<\frac{\epsilon}{k}
\end{align}
By triangle inequality, it can to see that the following lemma is
true.
\begin{lemma}\label{pseudo}
Let $S$ be a finitely generating set of $G$ and $\varphi:G\times
X\to X$ be a continuous action on compact metric space $(X,d)$. For
$\epsilon>0$ and $N\in\mathbb{N}$, there is $\delta>0$ such that if
$f:G\to X$ is a $\delta$-pseudo orbit, then for every $h\in G$ with
$||h||_S<N$ and every $g\in G$ we have $d(f(hg), \varphi(h,
f(g)))<\epsilon$
\end{lemma}
Subset $H\subseteq G$ is syndetic if there is finite set $F\subseteq
G$ such that $G=FH$. Hence a subgroup $H$ is syndetic in $G$ if and
only if it is finite index subgroup of $G$ i.e. there is finite set
$\{g_i\}_{i=1}^n$ such that $G= \bigcup_{i=1}^n g_iH$.
\begin{proposition}\label{syndetic}
Let $H$ be a finite index subgroup of $G$. Then continuous action
$\varphi:G\times X\to X$ has persistent shadowing property if
$\varphi:H\times X\to X$ has persistent shadowing property.
\end{proposition}
\begin{proof}
Let $H$ be finite index subgroup of $G$. Let $A$ be a symmetric
finitely generating set of $H$. We can add more elements to $A$ to
get a symmetric finitely generating set $S$ of $G$.  Also let $G=
\bigcup_{i=1}^n g_iH$ and $N=\max \{ ||g_i||_S: 1\leq i \leq n\}$.

 Let $\epsilon>0$ be given. Choose $\delta>0$ such that  for all $g\in G$ with $||g||_S<N$
\begin{align}\label{k}
d(x, y)<\delta\Rightarrow d(\varphi(g, x), \varphi(g,
y))<\frac{\epsilon}{N}
\end{align}
Choose $\eta>0$ corresponding to $\delta>0$ and $N\in\mathbb{N}$
that satisfies Lemma\ref{pseudo}.\\
 Let $\epsilon>0$ be given. By triangle inequality, it is easy to see that for $\epsilon>0$ and $N$ above, there is $\eta>0$
 such that if $d_S(\varphi, \psi)<\eta$, and $d(a, b)<\eta$, then

 \begin{equation}\label{b}
     d(\varphi(g, a), \psi(g, b))<\frac{\epsilon}{N},  \forall g\in G \text{ with } ||g||_{S} < N.
 \end{equation}
 and
 \begin{equation}\label{pl}
d(\psi(g, a), \psi(g, b))<\frac{\epsilon}{N}, \text{ for all }
||g||_S<N.
 \end{equation}
Let $\varphi:H\times X\to X$ has persistent shadowing property, we
show that $\varphi:G\times X\to X$ has persistent shadowing
property. Let $\epsilon>0$ be given. Choose $\delta>0$ that
satisfies Lemma \ref{pseudo} and Relation \ref{b}, Relation
\ref{pl}. Choose $\eta>0$ corresponding $\frac{\delta}{2}>0$  by
definition of persistent shadowing property of $\varphi|H$. We show
that every $\eta$-pseudo orbit $F:G\to X$ for continuous action
$\psi:G\times X\to X$ with $d_S(\varphi, \psi)<\eta$ can
$\psi$-shadowed by a point of $X$. The map $F:G\to X$ is
$2\eta$-pseudo orbit for $\varphi:G\times X\to X$, hence
    \begin{equation}\label{kl}
      d(F(gh), \varphi(g, F(h)))<\frac{\epsilon}{N}, \text{  for all } g \in G \text{ with } ||g||_S<N.
    \end{equation}
    Since $F:H\to X$ is $\eta$-pseudo orbit for $\psi:H\times X\to X$ with $d_A(\varphi, \psi)<\eta$,
    hence by persistent shadowing property of $\varphi|H$ there is $p\in X$ such that $d(F(h), \psi(h, p))<\frac{\delta}{2}$.
     Also by Relation \ref{pl} we have $d(\varphi(g_i, F(h)), \psi(g_i, \psi(h, p)))<\frac{\epsilon}{N}$. Hence by Relation \ref{kl}, we have
      $d(F(g_ih), \psi(g_ih, p))<\epsilon$ i.e. $d(F(g), \psi(g, p))<\epsilon$ for all $g\in G$.

\end{proof}
\begin{remark}\label{224}
Let $P$ be one of the following property:  $(a)$ shadowing property,
$(b)$ $\alpha$-persistent, $(c)$ $\beta$-persistent. Similar to
proof of Proposition \ref{syndetic}, if $H\leq G$ is syndetic subset
of $G$ and continuous action $\varphi:H\times X\to X$ does have $P$-
property, then $\varphi:G\times X\to X$ has $P$-property.
\end{remark}
\subsection{ The Action Of Free Groups}\label{s21}
Group  $BS(1, n)$ is solvable group, hence it is not free group. In
the case of finitely free group actions $G$ , by \cite[Theorem
4.9.]{osipov}, if $\varphi:G\times X\to X$ has shadowing property,
then $\varphi_g:X\to X$ has shadowing property, for all $g\in G$. In
the following, we extend it in the case of  persistent shadowing
property.
\begin{proposition}\label{op2}
Let $\varphi:G\times X\to X$ be a continuous action of free group
$F_2=\langle a, b\rangle$ on compact metric space $(X, d)$.
 If $\varphi:F_2\times X\to X$ has persistent shadowing property, then $\varphi_{a^{-1}b}:X\to X$ has persistent shadowing property.

\end{proposition}
\begin{proof}
  Let $\epsilon>0$ be given. Choose $\epsilon_0>0$ corresponding to $\epsilon>0$ by persistent shadowing property of $\varphi:G\times X\to X$. For $\epsilon_0>0$ there is $\delta>0$ such that for every continuous action $\psi:F_2\times X\to X$ that is $\delta$-close to $\varphi:G\times X\to X$, we have
        \begin{equation*}
          d(x, y)<\delta\Rightarrow d(\psi(g, x), \psi(g, y))<\epsilon, |g|_S\leq 2,
        \end{equation*}
        Let $\{x_n\}_{n\in\mathbb{Z}}$ be $\delta$ pseudo orbit of homeomorphism $f:X\to X$ wit $d(f, \varphi_{a^{-1}}\varphi_b)<\delta$.
         It is easy to see that if $\psi:F_2\times X\to X$ is generated by $\psi_a=\varphi_a$ and $\psi_b= \varphi_a\circ f$, then $\psi:F_2\times X\to X$
         is $\epsilon_0$- close to $\varphi:F_2\times X\to X$. Define $K: F_2\to X$ define by $K(t)= \psi(v, x_k)$ where $v\in G$ is an element with minimal
         length such that $t=v(a^{-1}b)^k$ for some $k\in \mathbb{Z}$. It is not hard to see that $K:G\to X$ is $\epsilon_0$- pseudo orbit of continuous
          action $\psi:F_2\times X\to X$ with $d(\varphi, \psi)<\epsilon_0$. By persistent shadowing property, there is $y\in X$ such that
          $d(K(g), \psi(g, y))<\epsilon$ for all $g\in G$. Since $K((a^{-1}b)^k)= x_k$ and $\psi((a^{-1}b)^k, y)= f^k(y)$, we have $d(f^k(y), x_k)<\epsilon$ for
          all $k\in \mathbb{Z}$.

\end{proof}

\begin{remark}\label{225}
Let $P$ be one of the following property:  $(a)$ shadowing property,
$(b)$ $\alpha$-persistent, $(c)$ $\beta$-persistent. Similar to
proof of Proposition \ref{op2}, we can show that if continuous
action $\varphi:F_2\times X\to X$ of free group $F_2=\langle a,
b\rangle$ on compact metric space $(X, d)$ has $P$- property, then
$\varphi_{a^{-1}b}:X\to X$ has $P$- property.
\end{remark}
\subsection{Persistent shadowing property and related
measure}\label{s22}
 For continuous action $\varphi:G\times X\to X$,
$\epsilon>0$, $\delta>0$ and generating set $S$, we denote
$PSh_{\varphi, S}(\delta, \epsilon)$  the set of all $x\in X$ such
that every $\delta$-pseudo orbit $f:G\to X$ of continuous action
$\psi$ with $d_S(\varphi, \psi)<\delta$ and $f(e)=x$, respect to
generating set $S$, can be $(\epsilon, \psi)$-shadowed by a point in
$X$.

It is clear that
\begin{enumerate}
\item Let $S, S'$ be generating sets for $G$ and $\epsilon>0$ be given. Then for every $\delta>0$ there is $\eta>0$ such that $PSh_{\varphi, S}(\eta, \epsilon)\subseteq PSh_{\varphi, S'}(\delta, \epsilon)$
  \item Continuous action $\varphi:G\times X\to X$ does have persistent shadowing property with respect to generating set $S$ if and only if for every $\epsilon>0$ there is $\delta>0$ such that $PSh_{\varphi, S}(\delta, \epsilon)=X$.
      \item If continuous action $\varphi:G\times X\to X$ does have persistent shadowing property on compact set $K\subseteq X$, then for every $\epsilon>0$ there exist a neighborhood $U$ of $K$ and  $\delta>0$ such that $U\subseteq PSh_{\varphi, S}(\delta, \epsilon)$.
          \item $PSh_{\varphi, S}(\delta, \epsilon)$ is a closed subset of $X$.
\end{enumerate}
The Borel $\sigma$-algebra of $X$ is the $\sigma$-algebra
$\mathcal{B}(X)$ generated by the open subsets of $X$.
 A Borel probability measure is a $\sigma$-additive measure $\mu$ defined in $\mathcal{B}(X)$ such that $\mu(X)=1$.
 We denote by $\mathcal{M}(X)$ the set of all Borel probability measures of $X$. This set is convex and  compact metrizable if it endowed with weak$^*$ topology:
 the one ruled by the convergence $\mu_n\to \mu$ if and only if $\int f d\mu_n\to \int f d\mu$ for every continuous map $f:X\to \mathbb{R}$.\\
\begin{definition}\label{plr}
 A measure $\mu\in\mathcal{M}(X)$ is compatible with the persistent
shadowing property  for the continuous action $\varphi:G\times X\to
X$, $\mu\in \mathcal{M}_{PSh}(X, \varphi)$ if for every $\epsilon>0$
there is $\delta>0$ such that if $\mu(A)>0$, then
\begin{equation*}
A\cap Sh_{\psi}(\delta, \epsilon)\neq \emptyset.
\end{equation*}
for every continuous action $\psi:G\times X\to X$ with
$d_S(\varphi,\psi)<\delta$.

\end{definition}

%
\begin{example}\label{exam1}
Let  continuous action $\varphi:G\times X\to X$ admits  an
$\varphi$-invariant measure and let $\varphi$  does have  persistent
shadowing property on the non-wandering set $\Omega(\varphi)$. Then
every $\varphi$-invariant Borel probability measure $\mu$ on $X$ is
compatible with the property of persistent shadowing property.
 \end{example}

 Let $\mu\in\mathcal{M}_{PSh}(X, \varphi)$ , $\epsilon>0$ and $h:(X, d)\to (Y, \rho)$ be a homeomorphism. We will show that there is $\delta>0$ such that
\begin{equation}
h_{*}(\mu)(B)>0 \Rightarrow B\cap PSh_{h\circ \varphi \circ
h^{-1}}(\delta, \epsilon)\neq \emptyset.
\end{equation}
For $\epsilon>0$ there is $\epsilon_0>0$ such that
\begin{equation}\label{rty}
d(a, b)<\epsilon_0\Rightarrow \rho(h(a), h(b))<\epsilon.
\end{equation}
For $\epsilon_0>0$ there is $\epsilon_1>0$ in definition of
$\mu\in\mathcal{M}_{PSh}(X, \varphi)$. Since $\mu(h^{-1}(B))>0$,
hence $h^{-1}(B)\cap \mathcal{M}_{PSh}(\epsilon_1, \epsilon_0)\neq
\emptyset$. For $\epsilon_1>0$ there is $\delta>0$ such that
\begin{equation}
\rho(c, d)<\delta\Rightarrow d(h^{-1}(c), h^{-1}(d))<\epsilon_1.
\end{equation}
 Fix $x\in h^{-1}(B)\cap PSh_\varphi(\epsilon_1, \epsilon_0)$. One can check that if
 $F':G\to Y$ be $\delta$-pseudo orbit of continuous action
 $\psi':G\times Y\to Y$ with $\rho_S(f\circ \varphi\circ f^{-1},
 \psi')<\delta$ and $F'(e)=h(x)$, then $F:G\to X$ defined by
 $F(g)=h^{-1}(F'(g))$ is $\epsilon_1$-pseudo orbit of continuous
 action $h^{-1}\circ \psi'\circ h $ with $d_S(\varphi,
 h^{-1}\circ \psi'\circ h)<\epsilon_1$ and $F(e)=x$. By $x\in h^{-1}(B)\cap PSh_\varphi(\epsilon_1,
 \epsilon_0)$, there is $p\in X$ such that $d(F(g), h^{-1}\circ
 \psi'_g\circ h(p))<\epsilon$ for all $g\in G$. By relation
 \ref{rty}, we have $d(h\circ F(g), \psi'_g(h(p)))<\epsilon$ i.e.
 $d(F'(g), \psi'_g(h(p)))<\epsilon$. This means that $h(x)\in B\cap
 PSh_{h\circ \varphi \circ h^{-1}}(\delta, \epsilon)$.
\begin{itemize}
\item  If $h:(X, d)\to(Y, \rho)$ is a homeomorphism and $\mu\in  \mathcal{M}_{PSh}(X, \varphi)$, then $h_{*}(\mu)\in \mathcal{M}_{PSh}(Y, h\circ \varphi\circ h^{-1})$,
\item If $G$ is an abelian group, then $\mathcal{M}_{PSh}(X, \varphi)$ is $(\varphi_g)_*$-invariant, for all $g\in G$.
\end{itemize}
It is known that every continuous action of a countable abelian group on compact metric space  admit  an $\varphi$-invariant measure. Also, One can check that  $\mathcal{M}_{PSh}(X, \varphi)$ is a convex  subset of $\mathcal{M}(X)$, hence by Proposition 2.12
in \cite{ali2}, the following item does hold.
\begin{itemize}
\item If $G$ is  an abelian group, then $\overline{\mathcal{M}_{PSh}(X, \varphi)}$ contains
a $\varphi$-invariant measure.
\end{itemize}
It is known that a  group action $\varphi:G\times X\to X$ on the
compact metric space $X$ has a $\varphi$- invariant Borel
probability measure on $X$ if and only if $G$ is amenable, see
\cite{1}. Hence, the following group action where  $G=SL(2, \mathbb{Z})$ and $X=\mathbb{R}\cup\{\infty\}$ does not contan any invarant meausre.\\
 
\begin{equation*}
  \varphi
  (\left(
  \begin{array}{ll}
  a & b\\
  c & d
  \end{array}
  \right),
  z)=\frac{az+b}{cz+d}.
\end{equation*}

%

 Take $$\mathcal{C}_{PSh(\varphi)}(\delta,
 \epsilon)=\{\mu\in\mathcal{M}(X): \mu(PSh_\varphi(\delta,
 \epsilon))=1\}.$$ It is easy to see that $\mathcal{C}_{PSh(\varphi)}(\delta,
 \epsilon)$ is a convex and closed subset of $\mathcal{M}(X)$ for
 every $\epsilon>0$ and any $\delta>0$.
 \begin{proposition}\label{kj}
 Let $\varphi:G\times X\to X$ be a continuous action. Then

 \begin{enumerate}
    \item $ \mathcal{M}_{PSh}(X, \varphi)=\bigcap_{n\in\mathbb{N}}(\bigcup_{m\in\mathbb{N}}(\bigcap_{l\in\mathbb{N}}C_{PSh(\varphi)}(n^{-1}+l^{-1}, m^{-1})))$

    \item The subset $\mathcal{M}_{PSh}(X, \varphi)$ is an
    $F_{\sigma\delta}$ subset of $\mathcal{M}(X)$.
 \end{enumerate}
 \end{proposition}
 \begin{proof}
 \begin{enumerate}
 \item Fix a $\mu\in\mathcal{M}_{PSh}(X, \varphi)$ and an $n\in\mathbb{N}$.
Choose a $\delta>0$ such that if $\mu(A)>0$ then $A\cap
PSh_\varphi(\delta, \frac{1}{n})\neq \emptyset$.Choose
$m\in\mathbb{N}$ such that $m^{-1}<\delta$. Note that if $A\cap
PSh_\varphi(\delta, \frac{1}{n})\neq \emptyset$, then $A\cap
PSh_\varphi(\frac{1}{m}, \frac{1}{n})\neq \emptyset$. This implies
that
$\mu\in\bigcup_{m\in\mathbb{N}}(\bigcap_{l\in\mathbb{N}}C_{PSh(\varphi)}(
m^{-1}, n^{-1}+l^{-1})))$. Conversely choose a
$\mu\in\bigcap_{n\in\mathbb{N}}(\bigcup_{m\in\mathbb{N}}(\bigcap_{l\in\mathbb{N}}C_{PSh(\varphi)}(
m^{-1}, n^{-1}+l^{-1})))$. Thus, for every $n\in\mathbb{N}$, there
is $k\in\mathbb{N}$ such that $\mu\in
\bigcap_{l\in\mathbb{N}}C_{PSh(\varphi)}(\frac{1}{k}, \frac{1}{n}+
\frac{1}{l})$. This means that for every $n\in\mathbb{N}$ there is
$k\in\mathbb{N}$ such that $\mu\in
\bigcap_{l\in\mathbb{N}}C_{PSh(\varphi)}(\frac{1}{k}, \frac{1}{n}+
\frac{1}{l})$. This implies that for every $\epsilon>0$ there exist
$N, K\in\mathbb{N}$ such that $(\frac{1}{N}+\frac{1}{K})<\epsilon$
and $\mu\in C_{PSh(\varphi)}(\frac{1}{N}+\frac{1}{K}, \frac{1}{K})$.
Therefor, for every $\epsilon>0$ choose $\delta=\frac{1}{K}$ to
conclude then $\mu\in\mathcal{M}_{PSh(\varphi)}(X)$.
\item Since $C_{PSh(\varphi)}(
m^{-1}, n^{-1}+l^{-1})))$ is a closed subset of $\mathcal{M}(X)$ and
a countable intersection of closed sets is closed, hence we can say
that  $(\bigcap_{l\in\mathbb{N}}C_{PSh(\varphi)}( m^{-1},
n^{-1}+l^{-1})))$ is a closed subset of $\mathcal{M}(X)$ for every
pair $m, n\in\mathbb{N}$. Therefor
$\bigcup_{m\in\mathbb{N}}(\bigcap_{l\in\mathbb{N}}C_{PSh(\varphi)}(
m^{-1}, n^{-1}+l^{-1})))$ is an $F_{\sigma\delta}$ subset of
$\mathcal{M}(X)$, for every $n\in\mathbb{N}$ and hence by item (1),
$\mathcal{M}_{PSh}(X,\varphi)$ is an $F_{\sigma\delta}$  subset of
$\mathcal{M}(X)$.
\end{enumerate}
 \end{proof}


\begin{proposition}\label{pki}
  Let $\varphi:G\times X\to X$ be a continuous action on compact metric space $(X, d)$ and $\mu\in\mathcal{M}_{PSh}(\varphi)$.
  Then $\varphi$ has persistent shadowing property on $supp(\mu)$.
\end{proposition}
\begin{proof}
 Let $\epsilon>0$ be given. Choose $0<\delta<\frac\epsilon2$ such that for every continuous action $\psi:G\times X\to X$ with $d_S(\varphi, \psi)<\delta$ we have
  \begin{equation*}
  \mu(A)>0\Rightarrow A\cap Sh_{\delta, \frac\epsilon2}(\psi)\neq \emptyset
\end{equation*}

For $\delta>0$ there is $0<\eta<\frac\delta2$ such that for every
continuous action $\psi:G\times X\to X$ with $d_S(\varphi,
\psi)<\eta$ we have
\begin{equation*}
  d(a, b)<\eta\Rightarrow d(\psi_s(a), \psi_s(b))<\frac\delta2, \forall s\in S,
\end{equation*}
We claim that if $d_S(\varphi, \psi)<\delta$ and $F:G\to X$ is an  $\eta$-pseudo orbit for   $\psi:G\times X\to X$ with  $F(e)=p\in supp(\mu)$, then there is $z\in X$ such that $d(F(g), \psi(g, z))<\epsilon$ for all $g\in G$. \\

By $p\in supp(\mu)$ we have $\mu(B_\eta(p))>0$, this implies that
$B_\eta(p)\cap S_{\delta, \epsilon}(\psi)\neq \emptyset$ for
$\psi:G\times X\to X$. Take $q\in B_\eta(p)\cap Sh_{\delta,
\epsilon}(\psi)$ and  define $f:G\to X$ by $f(e)=q$ and $f(g)=F(g)$
for all $g\neq e$. It is easy to see that $f:G\to X$ is a
$\delta$-pseudo orbit of $\psi$. By $f(e)\in Sh_{\delta,
\epsilon}(\psi)$, there is $y\in X$ with $d(f(g), \varphi(g,
y))<\frac\epsilon2$ for all $g\in G$. This implies that $d(F(g),
\psi(g, y))<\epsilon$ for all $g\in G$.

\end{proof}

It is known that the set of Borel probability measures of a compact
metric space $X$ with support equals to $X$ is a dense $G_\delta$
subset of $\mathcal{M}(X)$, see \cite[Lemma 3.6]{baut}, also if $X$
has no isolated point,  then the non-atomic Borel probability
measures is a dense $G_\delta$ subset of $\mathcal{M}(X)$ see
\cite[Corollary 8.2]{par}, thus since $X$ is a compact space, we can
say that if $X$ is a compact space without isolated point, then the
set of non-atomic Borel probability measures with support equals to
$X$ is dense in $\mathcal{M}(X)$.  Hence by Proposition \ref{pki},
we have
\begin{corollary}
  Let $\varphi:G\times X\to X$ be a continuous action of a compact metric space $X$ without
isolated point.If every non-atomic Borel probability measure $\mu$
is compatible with persistent shadowing property for
$\varphi:G\times X\to X$, then $\varphi$ has persistent shadowing
property.
\end{corollary}
For continuous actions $\varphi, \psi:G\times X\to X$,  and $x\in
X$, we denote

\begin{equation*}
\Gamma_\epsilon^{\varphi,\psi}(x)= \bigcap_{g\in G}\varphi (g^{-1},
B[\psi(g, x), \epsilon])=\{y\in X: d(\varphi(g, y), \psi(g, x))\leq
\epsilon \text{ for every } g\in G\}
\end{equation*}
and
\begin{equation*}
B(\epsilon, \varphi, \psi)=\{x\in X: \Gamma_\epsilon^{\varphi,
\psi}(x)\neq \emptyset\}.
\end{equation*}
It is easy to see that $B(\epsilon, \varphi, \psi)$ is a compact set
in $X$.


We say that
\begin{enumerate}
\item A measure $\mu\in\mathcal{M}(X)$ is compatible with the
shadowing property  for the continuous action $\varphi:G\times X\to
X$, $\mu\in \mathcal{M}_{Sh}(X, \varphi)$ if for every $\epsilon>0$
there is $\delta>0$ such that if $\mu(A)>0$, then
\begin{equation*}
A\cap Sh_{\varphi}(\delta, \epsilon)\neq \emptyset.
\end{equation*}
\item (\cite{ali2}) A measure $\mu\in\mathcal{M}(X)$ is compatible with the
$\alpha$-persistent  for the continuous action $\varphi:G\times X\to
X$, $\mu\in \mathcal{M}_\alpha(X, \varphi)$, if for every
$\epsilon>0$ there is $\delta>0$ such that if $\mu(A)>0$, then
\begin{equation*}
A\cap B(\epsilon, \varphi, \psi)\neq \emptyset.
\end{equation*}
for every continuous action $\psi:G\times X\to X$ with
$d_S(\varphi,\psi)<\delta$.
  \item A measure $\mu\in\mathcal{M}(X)$ is compatible with the
$\beta$-persistent  for the continuous action $\varphi:G\times X\to
X$, $\mu\in\mathcal{M}_\beta(X, \varphi)$, if for every $\epsilon>0$
there is $\delta>0$ such that if $\mu(A)>0$, then
\begin{equation*}
A\cap B(\epsilon, \psi, \varphi)\neq \emptyset.
\end{equation*}
for every continuous action $\psi:G\times X\to X$ with
$d_S(\varphi,\psi)<\delta$.
\end{enumerate}

\begin{remark}\label{pkii}
With similar proof of Proposition \ref{kj},  one can check that
$\mathcal{M}_{Sh}(X, \varphi)$, $\mathcal{M}_\alpha(X, \varphi)$ and
$\mathcal{M}_\beta(X, \varphi)$ are $F_{\sigma\delta}$ subsets of
$\mathcal{M}(X)$. Also  for homeomorphism $h:(X, d)\to (Y,\rho)$, if
$\mu\in \mathcal{M}_{Sh}(X, \varphi)$, $\mu\in
\mathcal{M}_{\alpha}(X, \varphi)$ and $\mu\in \mathcal{M}_{\beta}(X,
\varphi)$, then $f_{*}(\mu)\in\mathcal{M}_{Sh}(Y, h\circ
\varphi\circ h^{-1})$, $f_{*}(\mu)\in\mathcal{M}_\alpha(Y, h\circ
\varphi\circ h^{-1})$, $f_{*}(\mu)\in\mathcal{M}_\beta(Y, h\circ
\varphi\circ h^{-1})$, respectively. With similar technics in
Proposition \ref{pki}, we can show that if $\mu\in
\mathcal{M}_{Sh}(X, \varphi)$, $\mu\in \mathcal{M}_{\alpha}(X,
\varphi)$ and $\mu\in \mathcal{M}_{\beta}(X, \varphi)$, then
$\varphi:G\times X\to X$ does have shadowing property,
$\alpha$-persistent and $\beta$-persistent on $supp(\mu)$,
respectively.
\end{remark}

\section{Pointwise dynamic}
In this section we introduce persistent shadowable points, uniformly
$\alpha$-persistent, $\beta$-persistent points for a continuous
action $\varphi$. Also, we recall notions of shadowable points,
$\alpha$-persistent points and $\beta$-persistent points for
continuous action $\varphi:G\times X\to X$. This section consists of
3-subsection. In Subsection \ref{3001}, we study relations between
various of shadowable points. In Subsection \ref{400}, we study the
set of  persistent shadowable points and we give some properties of
it. Finally, in Subsection \ref{401}, we study the relation between
compatibility of a measure  with respect $P$-property and measure of
points in $X$ with $P$-property, where $P$-property can be
persistent shadowing property, shadowing property,
$\alpha$-persistent and $\beta$-persistent.

\subsection{Relation between various shadowable
points}\label{3001}
\begin{definition}
Let $S$ be a finitely generating set of $G$ and  $\varphi:G\times
X\to X$ be a continuous action.
\begin{enumerate}
\item (\cite{sang}) A point  $x\in X$ is called shadowable point  for $G$- action
$\varphi:G\times X\to X$, if for every $\epsilon>0$ there is
$\delta=\delta(\epsilon,x)>0$ such that for every $\delta$- pseudo
orbit $f:G\to X$ with $f(e)=x$ there is $p\in X$ such that $d(f(g),
\varphi(g, p))<\epsilon$ for all $g\in G$.
\item  A point $x\in X$ is $\alpha$-persistence  ( uniformly $\alpha$-persistent) for a continuous action $\varphi:G\times X\to X$
if for every $\epsilon>0$ there is $\delta_x>0$ such that  for every
continuous action $\psi:G\times X\to X$ with $d_S(\varphi,
\psi)<\delta$ ( and every $x'\in B_\delta(x)$)  there is $y\in X$
such that $d(\varphi(g, y), \psi(g, x))<\epsilon$ ( resp.
$d(\varphi(g, y), \psi(g, x'))<\epsilon$ ) for all $g\in G$.
Hereafter $Perssis_\alpha(\varphi)$ and $UPersis_\alpha(\varphi)$
will denote the set of all $\alpha$-persistence points  and
uniformly $\alpha$-persistent points  of $\varphi$, respectively.
 \item  A point $x\in X$ is $\beta$-persistence ( uniformly $\beta$-persistent) for a continuous action $\varphi:G\times X\to X$
if for every $\epsilon>0$ there is $\delta_x>0$ such that  for every
continuous action $\psi:G\times X\to X$ with $d_S(\varphi,
\psi)<\delta$ ( and every $x'\in B_\delta(x)$)  there is $y\in X$
such that $d(\varphi(g, x), \psi(g, y))<\epsilon$ ( resp.
$d(\varphi(g, x'), \psi(g, y))<\epsilon$) for all $g\in G$.
Hereafter $Perssis_\beta(\varphi)$ and $UPersis_\beta(\varphi)$ will
denote the set of all $\beta$-persistence points  and uniformly
$\beta$-persistent points  of $\varphi$, respectively.
\item  A point
$x\in X$ is called persistent  shadowable point for
$\varphi:G\times X\to X$, $x\in PSh(\varphi)$,  if for every
$\epsilon>0$ there is $\delta>0$ such that every $\delta$-pseudo
orbit $f:G\to X$ for $\psi:G\times X\to X$ with $d_S(\varphi,
\psi)<\delta$  and $f(e)=x$ can be $(\psi, \epsilon)$-shadowed by a
point.

\end{enumerate}

\end{definition}

It is easy to see that $PSh(\varphi)\subseteq Sh(\varphi)\subseteq
UPersis_\alpha(\varphi\subseteq Persis_\alpha(\varphi)$ and
$PSh(\varphi)\subseteq UPersis_\beta(\varphi)\subseteq
Persis_\beta(\varphi)$. The following example shows that the
converse of inclusions need not be hold.

\begin{example}\label{non-shadowable}

 \begin{enumerate}
   \item $Persis_\alpha\neq UPersis_\alpha(\varphi)$. Let $X=\mathbb{S}^1\cup\{(x, 0): -1\leq x\leq 1\}$. Since  $(-1, 0)$ is
a fixed point of every continuous action, then  $(1, 0)\in
Persis_\alpha(\varphi)$ where $\varphi:F_2\times X\to X$ is defined
by $\varphi(g, x)=x$. We claim that $(1, 0)\in
Persis_\alpha(\varphi)- UPersis_\alpha(\varphi)$. By contradiction,
we assume that $(1, 0)\in UPersis_\alpha(\varphi)$. For $\epsilon>0$
there is $\delta>0$ by $(1, 0)\in UPersis_\alpha(\varphi)$.
 Let $s_1:[-1,1]\to\mathbb{R}:x\mapsto\frac{\delta}{4}(1-|x|)$ and $s_2:[-1,1]\to\mathbb{R}:x\mapsto\frac{\delta}{6}(1-|x|)$  then $-1\le x-s_i(x)\le x+s_i(x)\le 1$ for any $x\in[-1,1]$, and $s_i(-1)=s_i(1)=0$, for $i=1, 2$. Define

$$g_i:X\to X:\langle x,y\rangle\mapsto\begin{cases}

\langle x-s_i(x),0\rangle,&\text{if }y=0\\

\left\langle x,y\right\rangle,&\text{otherwise}

\end{cases}$$

 For any $(x, y)\in \{(x, 0): -1<x<1 \}$, the first coordinate of $g_i((x, y))$  is less than $x$.
 Assume that $\psi:F_2\times X\to X$ is generated by $\varphi_a=g_1$ and $\varphi_b=g_2$. Then $d_S(\varphi, \psi)<\delta$, hence for
 $y\in (-1, 1)\times \{0\}$ with $d(y, (1, 0))<\delta$, there is $p\in X$ with $d(p, \psi(g, y))<\epsilon$ for all $g\in G$ that is a contradiction, because $g^k_i(y)\to (-1, 0)$ as $k\to \infty$.
   \item By \cite[Remark 4.4]{art1}, there is system $(X, f)$ such that $f:X\to X$ is $\alpha$-persistent while it does not shadowing property. Hence $UPersis_\alpha(f)=X$ but $Sh(f)\neq X$.
 \end{enumerate}
\end{example}

A point $x\in X$ is an equicontinuous point for continuous action
$\varphi:G\times X\to X$ if for every $\epsilon>0$ there is
$\delta_x>0$ such that
\begin{equation}
d(x, y)<\delta_x\Rightarrow d(\varphi(g, x), \varphi(g,
y))<\epsilon, \forall g\in G
\end{equation}
The set of equicontinuous points for $\varphi:G\times X\to X$ is
denoted by $Eq(\varphi)$.  It is easy to see that
\begin{equation}
Persis_\beta(\varphi)\cap Eq(\varphi)\subseteq
UPersis_\beta(\varphi)
\end{equation}
and

\begin{equation}\label{268}
Persis_\alpha(\varphi)\cap Eq(\varphi)\subseteq
Persis_\beta(\varphi)
\end{equation}

 Let $\varphi:G\times X\to X$ be continuous action as in
Example\ref{non-shadowable}. Then $\varphi$ is equcontinuous action
and $(1, 0)\in Persis_\beta(\varphi)=UPersis_\beta(\varphi)$ while
by Example \ref{non-shadowable}, $(1, 0)\notin
UPersis_\alpha(\varphi)$. This implies that
$UPersis_\beta(\varphi)\neq UPersis_\alpha(\varphi)$ and the
converse of Relation \ref{268} does not hold.\\
In the following, we show that on compact manifold $M$ without
boundary with $dim(M)\geq 2$, notions of shadowable point, uniform
$\alpha$-persistent point and $\alpha$-persistent point are
equivalent. It is not hard to see that $x$ is shadowable point if
and only if it is   finite shadowable point. We say that $x\in X$ is
finite shadowable point, if for every $\epsilon>0$ there is
$\delta>0$ such that for every $n\in \mathbb{N}$, every $\delta-n$-
pseudo orbit $f:G_n\to X$ with $f(e)=x$ can be $\epsilon$-shadowed
by point $p\in X$. Note that $G_n=\{g\in G: |g|_S\leq n\}$.

  \begin{definition}
We say that the space $X$ is generalized homogeneous, if for every
$\epsilon>0$ there exists $\delta > 0$ such that if $\{(x_1,
y_1),\ldots , (x_n, y_n)\}$ is a finite set of points in $X\times X$
satisfying:
\begin{enumerate}
    \item for every $i=1,\ldots,n, d(x_i,y_i)<\delta$,
    \item if $i\neq j$ then $x_i\neq x_j$ and $y_i\neq y_j$,
\end{enumerate}
then there is a homeomorphism $h:X\rightarrow X$ with $d_0(h,
id)<\epsilon$ and $h(x_i)=y_i$ for $i=1,\ldots,n$. \label{homogen}
\end{definition}
For example, a topological manifold $X$ without  boundary
($dim(X)\geq 2$), a Cartesian product of a countably infinite number
of manifolds with nonempty boundary and a cantor set are generalized
homogeneous \cite{pil}.
\begin{proposition}\label{u}
Let $X$ be a generalized homogeneous  compact metric space and
$\varphi:G\times X\rightarrow X$ be a continuous action. Then
$Sh(\varphi)= Persis_\alpha(\varphi)$.
\end{proposition}
\begin{proof}
  It is clear that $Sh(\varphi)\subseteq Persi_\alpha(\varphi)$. Let $x\in Persis_\alpha(\varphi)$ and $\epsilon$ be given. W claim that there is $\delta>0$ such that for every $n\in \mathbb{N}$, every $\delta-n$- pseudo orbit $f:G_n\to X$ with $f(e)=x$ can be $\epsilon$-shadowed by point $p\in X$. Choose $0<\epsilon_0<\frac{\epsilon}{2}$ corresponding to $\epsilon>0$ by $x\in Persis_\alpha(\varphi)$. Choose $0<\delta_0<\frac{\epsilon_0}{2}$ corresponding to $\epsilon_0>0$ by Definition \ref{homogen}. For $\delta_0>0$ there is $0<\delta<\frac{\delta_0}{2}$ such that
   \begin{equation*}
     d(a, b)<\delta\Rightarrow d(\varphi(g, a), \varphi(g, b))<\delta_0, \forall |g|_S\leq 2.
   \end{equation*}

   Let $f:G_n\to X$ be a $\delta-n$pseudo orbit with $f(e)=x$. Similar proof of Lemma 2.1.2 in \cite{pil}, we can construct $\delta_0$-pseudo orbit $F:G_n\to X$ with the following property:
   \begin{itemize}
     \item $F(e)=x$,
     \item $\varphi(s, F(g))\neq F(sg), \forall s\in S, \forall |g|_S<n$,
     \item $d(F(g),f(g))<\delta_0, \forall g\in G_n$
   \end{itemize}

   By Definition \ref{homogen}, for $\{ (\varphi(s, F(g)), F(sg)); g\in G_n\}$, there is a homeomorphism  $h: X\to X$ with $d_0(h, id)<\epsilon_0$ such that $h(\varphi(s, F(g)))= F(sg)$. Let $\psi:G\times X\to X$ be generated by $h\circ \varphi_s$ for $s\in S$. Then $\psi:G\times X\to X$ is $\epsilon_0$-close to $\psi:G\times X\to X$. This implies that there is $p\in X$ with $d(\psi(g, x), \varphi(g, p))<\epsilon$ for all $g\in G$. But $\psi(g, x)=F(g)$, hence $d(f(g), \varphi(g, p))\leq d(f(g), F(g))+ d(F(g), \varphi(g, p))<\epsilon$ for all $g\in G$.
\end{proof}

By Theorem 3.5 in \cite{sang}, a continuous action $\varphi:G\times
X\to X$ on compact metric space $(X, d)$ has shadowing property if
and only if it is pointwise shadowable point. Hence by Proposition
\ref{u}, we have

\begin{corollary}
 If $\varphi:G\times X\to X$ is a continuous action of finitely generated group $G$ on generalized homogeneous  compact metric space $(X, d)$, then the following equivalent:
 \begin{enumerate}
   \item $\varphi$ has shadowing property,
   \item $\varphi$ is pointwise shadowable point,
   \item $\varphi$ is pointwise $\alpha$-persistent,
   \item $\varphi$ is $\alpha$-persistent.
 \end{enumerate}
 \end{corollary}
Let $X$ be a compact metric space. It is easy to see that the
following properties hold.
\begin{enumerate}
    \item Continuous action $\varphi:G\times X\to X$ is
    $\alpha$-persistent if and only if $UPersis_\alpha(\varphi)=X$
    \item Continuous action $\varphi:G\times X\to X$ is
    $\beta$-persistent if and only if $UPersis_\beta(\varphi)=X$
    \item Equicontinuous action $\varphi:G\times X\to X$ is
    $\beta$-persistent if and only if $Persis_\beta(\varphi)=X$.

\end{enumerate}

\subsection{Some properties of persistent shadowable
points}\label{400}

In the following, we give some properties of persistent shadowable
points.

\begin{theorem}\label{wok}
  Let $S$ be a finitely generating set of $G$ and $\varphi:G\times X\to X$ be a continuous action on compact metric space $(X, d)$.
  \begin{enumerate}
    \item The point $x=p$ is persistent shadowable if and only if for every $\epsilon>0$ there is $\delta>0$ such that every $\delta$-pseudo orbit through $B[p, \delta]$ of a continuous action $\psi:G\times X\to X$ with $d_S(\varphi, \psi)<\delta$ can be shadowed by a $\psi$-orbit.
    \item Continuous action $\varphi:G\times X\to X$ has persistent shadowing property on compact set $K$ if and only if $K\subseteq PSh(\varphi)$.
    \item Continuous action $\varphi:G\times X\to X$  has the persistent shadowing property if and only if it is pointwise persistent shadowable.
    \item $PSh(\varphi)= UPersis_\beta(\varphi)\cap Sh(\varphi)$.
   \item Continuous action $\varphi:G\times X\to X$ has persistent shadowing property if and only if it is $\beta$-persistent and it has shadowing property.
  \end{enumerate}
\end{theorem}
\begin{proof}
  \begin{enumerate}
  \item   Suppose by contradiction that $x=p$ is persistent shadowable point but there are $\epsilon>0$, a sequence of continuous actions $\psi_k:G\times X\to X$ with $d_S(\varphi, \psi_k)\leq \frac{1}{k}$, a sequence of $\frac{1}{k}$-pseudo orbits $f^k:G\to X$ of $\psi_k:G\times X\to X$ with $d(f^k(e), x)\leq \frac{1}{k}$ such that $f^k:G\to X$ can not be $2\epsilon$-shadowed by any $\psi_k$-orbit, for every $k\in \mathbb{N}$. For this $\epsilon$ we let $\delta$ be given by the persistently shadowableness of $x$. We can assume that $\delta<\epsilon$. \\
  On the one hand,
  \begin{align*}
     d(\psi_k(s, p), \psi_k(s, f^k(e)) & \leq d(\psi_k(s,p), \varphi(s, p))+ d(\varphi(s, p), \varphi(s, f^k(e)))+ d(\varphi(s, f^k(e)), \psi_k(s, f^k(e)))\\
     & \leq 2d_s(\varphi, \psi_k)+ d(\varphi(s, p), \varphi(s, f^k(e)))
  \end{align*}
  We can choose $k$ large satisfying
  \begin{equation}\label{301}
    \max \{d(\psi_k(s, p), \psi_k(s, f^k(e)), \frac{1}{k}\}<\frac{\delta}{2}, \forall s\in S
  \end{equation}

Let us define $F^k:G\to X$ by
 \[ F^k(g)= \left\lbrace
  \begin{array}{c l}
     f^k(g), & \text{\rm{ $ g\neq e$}},\\
e, & \text{\rm{$g=e$}}.
  \end{array}
\right. \]

Then

 \[ d(\psi_k(s, F^k(g)), F^k(sg))= \left\lbrace
  \begin{array}{c l}
     d(\psi_k(s, f^k(g)), f^k(sg), & \text{\rm{ $ \text{ for } g\notin\{e, s^{-1}: s\in S\}$}},\\
d(\psi_k(s, f^k(g)), p), & \text{\rm{$ \text{ for } g=s^{-1}$}},\\
d(\psi_k(s, p), f^k(s)), & \text{\rm{$ \text{ for } g=e$}}.
  \end{array}
\right. \]

Since $f^k:G\to X$ is $\frac{1}{k}$-pseudo orbit of $\psi_k:G\times
X\to X$, hence
\begin{equation}\label{zx}
  d(\psi_k(s, F^k(g)), F^k(sg))<\frac{1}{k}<\delta, \text{ for } g\notin \{e, s^{-1}: s\in S\}
\end{equation}
Also inequality
\begin{equation*}
  d(\psi_k(s, f^k(s^{-1})), p)\leq d(\psi_k(s, f^{k}(s^{-1})), f^{k}(e))+d(f^k(e), p),
\end{equation*}
implies that

\begin{equation}\label{zc}
  d(\psi_k(s, F^k(g)), F^k(sg))<\delta, \text{ for } g=s^{-1}.
\end{equation}
By Relation \ref{301} and Relation \ref{zc}
\begin{equation*}
  d(\psi_k(s, p), f^k(s))\leq d(\psi_k(s, p)+\psi_k(s, f^k(e)))+ d(\psi_k(s, f^k(e)), f^k(s))
\end{equation*}
we have
\begin{equation}\label{zn}
  d(\psi_k(s, F^k(g)), F^k(sg))<\delta, \text{ for } g=e
\end{equation}
Hence by Relations \ref{zx},\ref{zc}, \ref{zn}, we get $F^k:G\to X$
is $\delta$-pseudo orbit of $\psi_k:G\times X\to X$. But
$d_S(\varphi, \psi_k)<\delta$, hence , by persistent shadowing
property of $\varphi$, for $\delta$-pseudo orbit $F^k:G\to X$ with
$F(e)=p$ of continuous action $\psi_k$ with $d_S(\varphi,
\psi_k)<\delta$, there is $z\in X$ such that $d(F^k(g), \psi_k(g,
z))<\epsilon$ for all $g\in G$. Since for $g\neq e$ one has
$d(f^k(g), \psi_k(g, z))=d(F^k(g), \psi_k(g, z))<\epsilon$ and for
$g=e$
\begin{equation*}
  d(z, f^k(e))\leq d(z, p)+ d(p, f^k(e))<\epsilon+\frac{1}{k}<2\epsilon,
\end{equation*}
Hence  $f^k:G\to X$ can be $2\epsilon$- shadowed by $\psi_k$-orbit
of $z\in X$. That is a contradiction.
\item It is sufficient to show that if $K\subseteq PSh(\varphi)$, then $\varphi$ has persistent shadowing property on $K$. Let $\epsilon>0$ be given. For every $x\in X$ there is $\delta_x>0$ corresponding to $\epsilon>0$ by item (2). Since $K$ is a compact space, the open cover $\{B[x, \delta_x]:x\in K\}$ has a finite open subcover $\{B[x_i, r_{x_i}]: i=1, 2, \ldots, n\}$.\\
  Take $\delta=\min \{ \delta_{x_i}: i\in \{1, 2, \ldots, n\}\}$ and let $F:G\to X$ be a $\delta$-pseudo orbit of continuous action $\psi:G\times X\to X$ with $d_S(\varphi, \psi)<\delta$ and $F(e)\in K$. Clearly $F(e)\in B[x_i, \delta_{x_i}]$ for some $1\leq i\leq n$. This implies that $F:G\to X$ is $\delta_{x_i}$-pseudo orbit through $B[x_i, \delta_{x_i}]$. Then $F:G\to X$ can be eventually $\epsilon$-shadowed by some $\psi$-orbit.
  \item Take  $K = X$ in item (3), we have that $\varphi$ has the
persistent shadowing property if and only if $PSh(\varphi)=X$.
   \item   Firstly, we show that $PSh(\varphi)\subseteq UPersis_\beta(\varphi)\cap Sh(\varphi)$. Take $x\in PSh(\varphi)$, $\epsilon>0$. Choose  $\delta_0>0$ corresponding $\frac\epsilon2>0$ by $x\in PSh(\varphi)$. Choose $\delta<\frac{\delta_0}{2}$ such that
       \begin{equation*}
         d(a, b)<\delta\Rightarrow d(\varphi(s, a), \varphi(s, b))<\frac{\delta_0}{2}.
       \end{equation*}
       Fix a continuous action $\psi:G\times X\to X$ with $d_S(\varphi, \psi)<\delta$. For $y\in B_\delta(x)$, define
       $F:G\to X$ by $F(g)=\varphi(g, y)$ if $g\neq e$ and $F(e)=x$, then $F:G\to X$ is a $\delta$-pseudo orbit of continuous action $\psi:G\times X\to X$ with $F(e)=x$. By $x\in PSh(\varphi)$, there is $p\in X$ such that $d(F(g), \psi(g, p))<\frac\epsilon2$ for every $g\in G$. This implies that $d(\varphi(g, y), \psi(g, p))<\epsilon$ for all $g\in G$. It follows that $x\in UPersis_\beta(\varphi)$. Since $PSh(\varphi)\subseteq Sh(\varphi)$, we get $x\in UPersis_\beta(\varphi)\cap Sh(\varphi)$. Therefore, $PSh(\varphi)\subseteq UPersis_\beta(\varphi)\cap Sh(\varphi)$.

       Now we show that $UPersis_\beta(\varphi)\cap Sh(\varphi)\subseteq PSh(\varphi)$. Suppose that   $x\in UPersis_\beta(\varphi)\cap Sh(\varphi)$ and  $\epsilon>0$ be given. We show that there is $\delta>0$ such that for every $\delta$-pseudo orbit $f:G\to X$ with $f(e)=x$ of $\psi$ such that $d_S(\varpi, \psi)<\delta$, there is $p\in X$ with $d(f(g), \psi(g, p))<\epsilon$ for all $g\in G$.
        Choose $\epsilon_0<\frac\epsilon4$ corresponding $\frac\epsilon2$ by $x\in UPersis_\beta(\varphi)$. There is $\eta>0$ corresponding $\frac{\epsilon_0}{2}$ by $x\in Sh(\varphi)$. Take $\delta<\frac\eta2$. If $f:G\to X$ is $\delta$-pseudo orbit of $\psi$ with $f(e)=x$ for continuous action $\psi:G\times X\to X$ with $d_S(\varphi, \psi)<\delta$, then $f:G\to X$ is $\eta$-pseudo orbit of $\varphi$ with $f(e)=x$. By $x\in Sh(\varphi)$ and $f(e)=x$, there is $y\in B_{\epsilon_0}(x)$ such that $d(f(g), \varphi(g, y))<\frac\epsilon2$ for all $g\in G$. Also by  $x\in UPersis_\beta(\varphi)$ and $y\in B_{\epsilon_0}(x)$ there is $p\in X$ such that $d(\varphi(g, y), \psi(g, p))<\frac\epsilon2$ for all $g\in G$. This implies that $d(f(g), \psi(g, p))<\epsilon$ for all $g\in G$.
.
   \item It is clear that persistent shadowing property implies $\beta$-persistent and shadowing property. For the converse let $\varphi$ is $\beta$-persistent
    and it has shadowing property, then by item (4), $PSh(\varphi)= Persis_\beta(\varphi)\cap Sh(\varphi)= X$, hence $\varphi$ is pointwise persistent shadowple and by item (3), $\varphi$ has persistent shadowing property.
\end{enumerate}
\end{proof}
\subsection{Various shadowable points and related
measures}\label{401}
 Item 1 in Proposition \ref{wok} implies that if
continuous action $\varphi:G\times X\to X$ does have persistent
shadowing property on compact set $K\subseteq X$, then for every
$\epsilon>0$ there
 exist a neighborhood $U$ of $K$ and  $\delta>0$ such that $U\subseteq PSh_{\varphi}(\delta, \epsilon)$. Also, by Item 2 in Proposition \ref{wok}, $K\subseteq PSh(\varphi)$
 implies that  continuous action
$\varphi:G\times X\to X$ does have persistent  shadowing property on
compact set $K\subseteq X$. Hence we have the following proposition.

\begin{proposition}\label{rtg}
Let $\varphi:G\times X\to X$  be a continuous action on compact
metric space $(X, d)$. Let $K\subset PSh(\varphi)$ be a compact
subset. Then for every $\epsilon>0$ there exist a neighborhood $U$
of $K$  and $\delta>0$  such that $U\subseteq PSh_{\varphi}(\delta,
\epsilon)$.

\end{proposition}

 One can check that Proposition \ref{rtg} is true in the case of
 shadowing property, $\alpha$-persistent and $\beta$-persistent.

 \begin{proposition}
 If $\mathcal{X}\in \{ Sh(\varphi),
 UPersis_\alpha(\varphi), UPersis_\beta(\varphi)\}$ and $K\subseteq
 X$ be a compact set. Then for every $\epsilon>0$ there exist a
 neighborhood $U$ of $K$ and $\delta>0$ such that $U\subseteq
 \mathcal{X}_{\varphi}(\delta, \epsilon)$.
 \end{proposition}
Assume that $supp(\mu)\subset PSh(X, \varphi)$ and $X$ be a compact
metric space. Since $supp(\mu)$ is a compact set, hence by
Proposition \ref{rtg}, for every $\epsilon>0$ there is $\delta>0$
such that
\begin{equation}
A\cap supp(\mu)\neq\emptyset \Rightarrow A\cap PS_{\delta,
\epsilon}(X, \varphi)\neq\empty.
\end{equation}
Hence we have the following corollary
\begin{equation}\label{pkiii}
supp(\mu)\subseteq PSh(\varphi)\Rightarrow
\mu\in\mathcal{M}_{PSh}(X, \varphi)
\end{equation}
By Proposition \ref{pki} and Remark \ref{pkii} and with similar
proof of Remark \ref{pkiii} we have the following proposition.
\begin{proposition}\label{Lb}
Let $\varphi:G\times X\to X$ be a continuous action of finitely
generated group on compact metric space $(X, d)$. Then
\begin{enumerate}
    \item $\mu\in M_{PSh}(X, \varphi)\Leftrightarrow
    supp(\mu)\subseteq PSh(\varphi)$,
    \item $\mu\in M_{Sh}(X, \varphi)\Leftrightarrow
    supp(\mu)\subseteq Sh(\varphi)$,
    \item $\mu\in M_{\alpha}(X, \varphi)\Leftrightarrow
    supp(\mu)\subseteq UPersis_\alpha(\varphi)$,
    \item $\mu\in M_{\beta}(X, \varphi)\Leftrightarrow
    supp(\mu)\subseteq UPersis_\beta(\varphi)$
\end{enumerate}
\end{proposition}
By Proposition \ref{kj}, the set of persistent shadowable points is
measureable. With similar proof, one can check that $Sh(\varphi)$,
$UPersis_\alpha(\varphi)$ and $UPersis_\beta(\varphi)$ are
measureable sets. Assume that $supp(\mu)\subseteq
\overline{PSh(\varphi)}$, then by Lemma 2.8 in \cite{boom}, if $X$
is a compact metric space, then there is a sequence
$\mu_n\in\mathcal{M}(X)$ with $supp(\mu_n)\subseteq PSh(\varphi)$
and converging to $\mu$ with respect to the $weak^*$ topology. By
Proposition \ref{Lb}, $\mu_n\in\mathcal{M}_{PSh}(X, \varphi)$. This
implies that $\mu\in \overline{\mathcal{M}_{PSh}(X, \varphi)}$.
Hence, we have  following relation
\begin{equation}\label{27}
 \mu(\overline{PSh(\varphi)})=1\Rightarrow \mu\in\overline{\mathcal{M}_{PSh}(X, \varphi)}.
\end{equation}
Conversely, let $\mu\in\overline{\mathcal{M}_{PSh}(X, \varphi)}$.
Choose $\mu_n\in \mathcal{M}_{PSh}(X, \varphi)$ such that $\mu_n\to
\mu$.  By inequality
$$\limsup_{n\to\infty}\mu_n(\overline{PSh(\varphi)})\leq
\mu(\overline{PSh(\varphi)})$$ and
$\mu_n(\overline{PSh(\varphi)})=1$, we have
$\mu(\overline{PSh(\varphi)})=1$. Hence we have the following
relation.
\begin{equation}\label{277}
\mu\in\overline{\mathcal{M}_{PSh}(X, \varphi)}\Rightarrow
\mu(\overline{PSh(\varphi)})=1.
\end{equation}
By Relation \ref{27} and Relation \ref{277}, we have the following
proposition.
\begin{proposition}\label{278}
Let $\varphi:G\times X\to X$ be a continuous action on compact
metric space $(X, d)$. Then $\mu(\overline{PSh(\varphi)})=1$ if and
only if $\mu\in\overline{\mathcal{M}_{PSh}(X, \varphi)}$.
\end{proposition}
With similar proof, we have the following proposition.
\begin{proposition}\label{Lbb}
Let $\varphi:G\times X\to X$ be a continuous action on compact
metric space $(X, d)$. Then
\begin{enumerate}
    \item $\mu(\overline{PSh(\varphi)})=1\Leftrightarrow \mu\in\overline{\mathcal{M}_{PSh}(X,
    \varphi)}.$
    \item $\mu(\overline{Sh(\varphi)})=1\Leftrightarrow \mu\in\overline{\mathcal{M}_{Sh}(X,
    \varphi)}.$
    \item $\mu(\overline{UPersis_\alpha(\varphi)})=1\Leftrightarrow \mu\in\overline{\mathcal{M}_\alpha(X,
    \varphi)}.$
    \item $\mu(\overline{UPersis_\beta(\varphi)})=1\Leftrightarrow \mu\in\overline{\mathcal{M}_\beta(X,
    \varphi)}.$
\end{enumerate}
\end{proposition}
 We claim that if
$PSh(\varphi)$ is a closed set in $X$, then
$\overline{\mathcal{M}_{PSh}(X,
    \varphi)}= \mathcal{M}_{PSh}(X,
    \varphi)$. Take $\mu_n\in\mathcal{M}_{PSh}(X, \varphi)$ with
    $\mu_n\to \mu$. Then $\limsup_{n\to
    \infty}\mu_n(PSh(\varphi))\leq
    \mu(PSh(\varphi))$ implies that
    $\mu(PSh(\varphi))=1$. Hence $supp(\mu)\subseteq
    PSh(\varphi)$. By Proposition \ref{Lb},
    $\mu\in\mathcal{M}_{PSh}(\varphi)$ i.e.
\begin{equation}\label{2312}
    \overline{PSh(\varphi)}=PSh(\varphi)\Rightarrow\overline{\mathcal{M}_{PSh}(X, \varphi)}= \mathcal{M}_{PSh}(X, \varphi).
\end{equation}

Conversely, let $\overline{\mathcal{M}_{PSh}(X,
    \varphi)}= \mathcal{M}_{PSh}(X,
    \varphi)$, we claim that $\overline{PSh(\varphi)}=PSh(\varphi)$.
    If it is not true, then there exist $\{x_n\}\subseteq
    PSh(\varphi)$ with $x_n\to x$ such that $x\notin PSh(\varphi)$.
    By $x_n\to x$, we have $m_{x_n}\to m_x$.
 Where  $m_t$ the Dirac measure supported on $t\in X$, indeed,
$m_t(A)=0$ or $1$ depending on whether $t\notin A$ or $t\in A$. It
is easy to see that

$$PSh(\varphi)=\{t\in X:m_t\in\mathcal{M}_{PSh}(X, \varphi)\}$$
 By $\{x_n\}\subseteq
    PSh(\varphi)$, we have $m_{x_n}\in\mathcal{M}_{PSh}(X,
    \varphi)$. Also, by $m_{x_n}\to m_x$ and $\overline{\mathcal{M}_{PSh}(X,
    \varphi)}= \mathcal{M}_{PSh}(X,
    \varphi)$, we have $x\in PSh(\varphi)$ which is a contradiction.
    This implies the following relation.

    \begin{equation}\label{12321}
\overline{\mathcal{M}_{PSh}(X,
    \varphi)}= \mathcal{M}_{PSh}(X,
    \varphi)\Rightarrow \overline{PSh(\varphi)}=PSh(\varphi).
    \end{equation}
    By Relation \ref{2312} and Relation \ref{12321}, we have the
    following proposition.
    \begin{proposition}\label{12354}
    Let $\varphi:G\times X\to X$ be a continuous action on compact
    metric space $(X, d)$. Then
    $$\overline{\mathcal{M}_{PSh}(X,
    \varphi)}= \mathcal{M}_{PSh}(X,
    \varphi)\Leftrightarrow \overline{PSh(\varphi)}=PSh(\varphi).$$
    \end{proposition}

 On can check that result of Proposition \ref{12354} can be obtain
 for other type of shadowing,indeed we have the following
 proposition.

   \begin{proposition}\label{12355}
    Let $\varphi:G\times X\to X$ be a continuous action on compact
    metric space $(X, d)$. Then
    \begin{enumerate}
        \item $\overline{\mathcal{M}_{PSh}(X,
    \varphi)}= \mathcal{M}_{PSh}(X,
    \varphi)\Leftrightarrow \overline{PSh(\varphi)}=PSh(\varphi).$
        \item $\overline{\mathcal{M}_{Sh}(X,
    \varphi)}= \mathcal{M}_{Sh}(X,
    \varphi)\Leftrightarrow \overline{Sh(\varphi)}=Sh(\varphi).$
        \item $\overline{\mathcal{M}_{\alpha}(X,
    \varphi)}= \mathcal{M}_{\alpha}(X,
    \varphi)\Leftrightarrow \overline{Persis_\alpha(\varphi)}=Persis_\alpha(\varphi).$
        \item $\overline{\mathcal{M}_{\beta}(X,
    \varphi)}= \mathcal{M}_{\beta}(X,
    \varphi)\Leftrightarrow \overline{Persis_\beta(\varphi)}=Persis_\beta(\varphi).$
    \end{enumerate}

    \end{proposition}

If  $\varphi:G\times X\to X$ is equicontinuous
    action, then $UPersis_\beta(\varphi)=Persis_\beta(\varphi)$
is a closed subset of $X$. This implies the following proposition.

\begin{proposition}\label{cin}
Let $\varphi:G\times X\to X$ be a equicontinuous action of a
finitely generated  group $G$ on compact metric space. Then
\begin{enumerate}
    \item $\overline{\mathcal{M}_{\beta}(X,
    \varphi)}= \mathcal{M}_{\beta}(X,
    \varphi)$
    \item $\mu(Persis_\beta(\varphi))= 1$ if and only if $\mu\in\mathcal{M}_\beta(X,
    \varphi)$.
\end{enumerate}
\end{proposition}
 Proof of the following proposition is clear.
\begin{proposition}\label{diracm}
Let $\varphi:G\times X\to X$ be a continuous action.
\begin{enumerate}
    \item $PSh(\varphi)=\{x\in X: m_x\in\mathcal{M}_{PSh}(X, \varphi)\}$.
    \item $Sh(\varphi)=\{x\in X: m_x\in\mathcal{M}_{Sh}(X, \varphi)\}$
    \item $Persis_\beta(\varphi)=\{x\in X: m_x\in\mathcal{M}_{\beta}(X, \varphi)\}$
    \item $Persis_\alpha(\varphi)=\{x\in X: m_x\in\mathcal{M}_{\alpha}(X, \varphi)\}$
\end{enumerate}

\end{proposition}
Assume that $\overline{PSh(\varphi)}=X$. By Theorem 6.3 in
\cite{par}, if $X$ is a separable metric space, then Then the set of
all measures whose supports are finite subsets of $PSh(\varphi)$ is
dense in $\mathcal{M}(X)$. Also by Lemma 2.7 in \cite{boom}, every
measure with finite support and supported on $PSh(\varphi)$ is a
finite convex combination of Dirac measures supported on points of
$PSh(\varphi)$. By Proposition \ref{diracm}, such Dirac measures are
compatible with persistent shadowing property. But finite convex
combination of measures in $\mathcal{M}_{PSh}(X, \varphi)$ is
compatible with persistent shadowing property. This implies that if
$\overline{PSh(\varphi)}=X$, then the set of finite convex
combination of measures in $\mathcal{M}_{PSh}(X, \varphi)$ is dense
in $\mathcal{M}(X)$. This means that if $\overline{PSh(\varphi)}=X$,
then $\overline{\mathcal{M}_{PSh}(X,\varphi)}=\mathcal{M}(X)$.
Conversely, assume that
$\overline{\mathcal{M}_{PSh}(X,\varphi)}=\mathcal{M}(X)$. We claim
that $\overline{PSh(\varphi)}=X$. If it is not true, then there is
$x\in X$ with $x\notin \overline{PSh(\varphi)}$. Choose open set $U$
of $x$ with $x\in U\subseteq X-\overline{PSh(\varphi)}$. Since
$m_x\in \overline{\mathcal{M}_{PSh}(X, \varphi)}$, there is a
sequence $\mu_n\in\mathcal{M}_{PSh}(X, \varphi)$ such that $\mu_n\to
m_x$. By Proposition \ref{Lb}, we have $Supp(\mu_n)\subseteq
PSh(\varphi)$. By $ U\subseteq X-\overline{PSh(\varphi)}$, we have
$\mu_n(U)=0$ for all $n\in \mathbb{N}$. There for $0=\liminf_{n\to
\infty}\mu_n(U)\geq m_x(U)=1$ which is a contradiction. With similar
technics we can prove the following propositions.
\begin{proposition}\label{3214}
Let $\varphi:G\times X\to X$ be a continuous action of a finitely
generated group $G$ on compact metric space $(X, d)$. Then the
following conditions hold.
\begin{enumerate}
    \item $\overline{\mathcal{M}_{PSh}(X, \varphi)}=\mathcal{M}(X)
    \Leftrightarrow \overline{PSh(\varphi)}=X$
    \item $\overline{\mathcal{M}_{Sh}(X, \varphi)}=\mathcal{M}(X)
    \Leftrightarrow \overline{Sh(\varphi)}=X$
    \item $\overline{\mathcal{M}_\beta(X, \varphi)}=\mathcal{M}(X)
    \Leftrightarrow \overline{Persis_\beta(\varphi)}=X$
    \item $\overline{\mathcal{M}_{\alpha}(X, \varphi)}=\mathcal{M}(X)
    \Leftrightarrow \overline{Persis_\alpha(\varphi)}=X$
\end{enumerate}
\end{proposition}

\section*{acknowledgments} The author wishes to thank Professor Morales   for his idea about Theorem \ref{wok} given in \cite{morales2}.

\bibliographystyle{amsplain}

\end{document}